\documentclass[letterpaper,12pt]{amsart}

\usepackage{enumerate}

\usepackage{color}
\usepackage{ulem,ifthen,xcolor,xkeyval,pdfcolmk}
\newcommand{\vol}{{\rm vol}\,}
\newcommand{\pa}{\partial}

\usepackage[abbrev]{amsrefs}
\usepackage{amssymb,amsmath,amsthm}

\newtheorem{thm}{Theorem}[section]
\newtheorem{prop}[thm]{Proposition}

\newtheorem{Def}[thm]{Definition}
\newtheorem{lem}[thm]{Lemma}

\newtheorem{remark}[thm]{Remark}
\newtheorem{corl}[thm]{Corollary}

\newcommand{\ie}{\emph{i.e.}}

\newcommand{\eps}{\varepsilon}

\def \<{\langle}
\def \>{\rangle}

\def \H{{\cal H}}

\def \H^0{{\cal H}^0 or}

\def \p{\partial}
\def \n{\nabla}
\def \beq{\begin{equation}}
\def \eeq{\end{equation}}

\def \n{\nabla}

\def \eref{\eqref}

\begin{document}



\title[weighted manifolds]{Heat kernel estimates and the essential spectrum on weighted manifolds}

\author{Nelia Charalambous}
\address{Department of Mathematics and Statistics, University of Cyprus, Nicosia, 1678, Cyprus} \email[Nelia Charalambous]{nelia@ucy.ac.cy}

\author{Zhiqin Lu} \address{Department of
Mathematics, University of California,
Irvine, Irvine, CA 92697, USA} \email[Zhiqin Lu]{zlu@uci.edu}

\thanks{The second author is partially supported by the DMS-12-06748.}
 \date{November 12, 2012}



\begin{abstract}
We consider a complete noncompact smooth Riemannian manifold $M$ with a weighted measure and the associated drifting Laplacian. We demonstrate that whenever the $q$-Bakry-\'Emery Ricci tensor on $M$ is bounded below,  then we can obtain an upper bound estimate for the heat kernel of the drifting Laplacian from the upper bound estimates of the heat kernels of the Laplacians on a family of related warped product spaces. We apply these results to study the essential spectrum of the drifting Laplacian on $M$.
\end{abstract}

\maketitle

\tableofcontents

\bigskip

\section{Introduction}
We consider a  complete noncompact smooth metric measure space $(M^n,g,e^{-f} dv)$, where $(M,g)$ is a complete noncompact smooth Riemannian manifold with a weighted volume measure $d\mu= e^{-f} dv$ such that $f$ is a smooth function on $M$ and $dv$ is the Riemannian measure. In this paper, we refer to such  a space as a {\it weighted manifold}.

The associated drifting Laplacian to such a weighted manifold is
\[
\Delta_f=\Delta -\nabla f \cdot \nabla,
\]
where $\Delta$ is the Laplace operator and $\nabla$ is the gradient  operator on  the Riemannian manifold $M$. $\Delta_f$ can be extended to a densely defined, self-adjoint, nonpositive definite operator on the space of square integrable functions with respect to the measure $e^{-f}dv$. The Bakry-\'Emery Ricci curvature is given by
\[
\textup{Ric}_f=\textup{Ric}+ \nabla^2 f,
\]
where $\textup{Ric}$ is the Ricci curvature of the Riemannian manifold and $\nabla^2 f$ is the Hessian of the function $f$. For a positive number $q$, the $q$-Bakry-\'Emery Ricci tensor is defined as
\[
\textup{Ric}_f^q=\textup{Ric}+ \nabla^2 f-\frac 1q\n f\otimes\n f=\textup{Ric}_f-\frac 1q\n f\otimes\n f.
\]

We use $\langle\,,\,\rangle$ to denote the inner product  of  the Riemannian metric  and $|\cdot|$ to denote the corresponding norm. Throughout this paper, we shall use the following version of the Bochner formula with respect to the drifting Laplacian.
\begin{equation} \label{fBochner}
\Delta_f|\nabla u|^2=2 |\n^2 u|^2+2\langle\nabla u, \nabla\Delta_f u\rangle + 2 \textup{Ric}_f(\nabla u, \nabla u),
\end{equation}
where $\nabla^2 u$ is the Hessian of $u$ and $|\n^2 u|^2=\sum u_{ij}^2$.

When $f=0$, the above formula is the {usual Bochner formula}  and was used by Li and Yau in ~\cite{LiYau} for their seminal work on the heat equation of Schr\"odinger operators on a Riemannian manifold.
One of the key inequalities used in their proof is the Cauchy inequality $|\nabla^2u|^2\geq (\Delta u)^2/n$. However, if $f\neq 0$, we do not have an analogous   relationship between the Hessian of $u$ and the {\it drifting} Laplacian of $u$. Consequently, there is more subtlety in obtaining the gradient estimate in the $f\neq 0$ case.

In 1985, Bakry and \'Emery had demonstrated that in the case of the drifting Laplacian the relevant Ricci tensor for obtaining a gradient estimate is the $q$-Bakry-\'Emery Ricci tensor~\cite{BE}. By generalizing the Li-Yau  method, Qian was able to prove a Harnack inequality and heat kernel estimates for the drifting Laplacian whenever $\textup{Ric}_f^q\geq 0$~\cite{Qi}.

In this article we will consider a weighted manifold $M$ on which, for some positive integer $q$,
 the $q$-Bakry-\'Emery Ricci tensor is bounded below by a nonpositive constant\footnote{The recent work of Munteanu-Wang~\cites{MuW,MuW2} shows that assuming $\textup{Ric}_f^q$ bounded below  may be too strong  in certain cases. On the other hand, we will see that only assuming $\textup{Ric}_f$ has  a lower bound seems to be too weak. It is therefore an interesting question to find the  appropriate assumptions  in-between  these two. We shall remark on the relevant results in this paper.}. We will associate to $M$ a family of warped products $\tilde M_\eps$. We shall show that the geometric analysis results on  $M$ are closely related to those  on $\tilde M_\eps$ by directly comparing the geometry of $M$ to that of $\tilde M_\eps$.  In this way, we  are  able to extend  previous work in this area   under the sole assumption of ${\rm Ric}_f^q$ bounded below, without any further constraints on the behavior of the function $f$. In particular, we are able to get the heat kernel estimate for the drifting Laplacian from the corresponding estimates in   the Riemannian case effortlessly. The proofs reveal the strong geometric connection of $M$ to the warped product spaces $\tilde M_\eps$. At the same time, they further illustrate the fact that the drifting Laplacian and $q$-Bakry-\'Emery Ricci tensor are projections (in some sense) of the  Laplacian and Ricci tensor of a higher dimensional space.

In the last two sections we apply the previous results to study the essential spectrum of the drifting Laplacian on a weighted manifold.  In Section~\ref{sec7} we will prove that the $L^p$ essential spectrum of the drifting Laplacian is independent of $p$, for all $p\in[1,\infty]$, whenever $\textup{Ric}_f^q$ is bounded below and the weighted volume of the manifold  grows uniformly subexponentially. In Section~\ref{sec6} we demonstrate that for $p\in[1,\infty]$ the $L^p$ essential spectrum of the drifting Laplacian is the nonnegative real line whenever $\textup{Ric}_f^q$ is asymptotically nonnegative.

Finally, we remark that for our heat kernel estimate it was only necessary to assume that the $q$-Bakry-\'Emery Ricci tensor was bounded below. Our upper bound depends on $q$ and when $q\to\infty$ it does not converge. This further confirms that a lower bound on $\textup{Ric}_f$ {alone} is not sufficient for obtaining heat kernel estimates. However, the classical proof of heat kernel bounds for the Laplacian can be generalized whenever ${\rm Ric}_f$ is bounded below under the additional assumption that the gradient of $f$ is bounded.\footnote{See Remark~\ref{rmk510} for details.}\\

{\bf Acknowledgement.} We would like to thank X.-D. Li, O. Munteanu, J.-P. Wang and D.-T. Zhou for the helpful discussions during the preparation of this paper.

\section{Gradient  and Heat Kernel Estimates on Riemannian Manifolds}\label{s2}

In this section we consider a complete noncompact Riemannian manifold $\tilde M$  of dimension $n+q$. We will review certain analytic properties of solutions to the heat equation on  $\tilde M$.   {Let $\tilde \Delta$ be the Laplacian of $\tilde M$}. Denote $\widetilde  {\rm Ric}$ the Ricci curvature tensor on $\tilde M$, and $\tilde B_{\tilde x_o}(r)$  the ball of radius $r$ at a fixed point $\tilde x_o$. The following gradient estimate for positive solutions to the heat equation is well known:
\begin{thm}\cite{SchoenYau_bk}*{Chapter IV, Theorem 4.2} \label{r1}
Let $R>0$ be a large number and assume that $\widetilde  {\rm Ric}\geq -K$ with $K\geq 0$ on $\tilde B_{\tilde x_o}(4R)$. Let $u(x,t)$ be a positive solution to the heat equation
\[
u_t -\tilde\Delta u=0
\]
on $\tilde B_{\tilde x_o}(4R)$. Then for any $\alpha>1$, $u$ satisfies the Li-Yau estimate
\[
\frac{|\n u|^2}{u^2}-\alpha \frac{u_t}{u} \leq  \frac{(n +q) \alpha^2}{2t} +    \frac{(n+q) \, K \, \alpha^2}{2(\alpha-1)} +C(n+q)\frac{\alpha^2}{R^2}\left(\frac{\alpha^2}{\alpha-1}+R\sqrt K\right),
\]
on $\tilde B_{\tilde x_o}(2R),$  where $C(n+q)$ is a constant that only depends on $n+q$.
\end{thm}

The following theorem is a localized version\footnote{That is, given that we assume the Ricci curvature condition {on} a ball, we obtain the inequality on a ball of smaller radius.} of \cite{SchoenYau_bk}*{Chapter IV, Theorem 4.5}.

\begin{thm} \label{T3}
Under the same assumptions as the previous theorem, for any $\alpha>1$  {and $x, y \in \tilde B_{\tilde x_o}(R)$,} the positive solution $u$ satisfies the Harnack inequality
\[
u(x,t_1) \leq   u(y,t_2) \bigl( \frac{t_2}{t_1} \bigr)^{\frac{(n +q) \alpha}{2}}  \cdot \exp { [\,\frac{\alpha \, \tilde{d}^2(x,y)}{4(t_2-t_1)} + A(R)\,(t_2-t_1)\, ] }
\]
where $0<t_1<t_2<\infty$, $\tilde{d}(x,y)$ is the distance between $x$ and $y$ in  $\tilde M$ and
\[
A(R)= \frac{(n+q) \, K \, \alpha}{2(\alpha-1)} +C(n+q)\frac{\alpha}{R^2}\left(\frac{\alpha^2}{\alpha-1}+R\sqrt K\right)
\]
where $C(n+q)$ is a constant  that only depends on $n+q$.
\end{thm}

With this Harnack inequality, one can prove an upper {bound} estimate for the heat kernel of the Laplacian as in \cite{SchoenYau_bk}*{Chapter IV}. The estimate we provide below is a localized version of  this result with an additional factor of   $e^{-\lambda_1(\tilde \Omega)t}$ where $\lambda_1(\tilde \Omega)$ is the first Dirichlet eigenvalue of the Laplacian on $\tilde \Omega$. This additional factor can be obtained using Davies' technique  as in \cite{Dav1} (see also \cite{LiNotes}).

\begin{thm} \label{thm24}
Let $R>0$ be a large number and assume that $\widetilde  {\rm Ric}\geq -K$ with $K\geq 0$ on $\tilde B_{\tilde x_o}(4R)$. Let $\tilde \Omega$ be a domain of $\tilde M$ such that $\tilde\Omega\supset \tilde B_{\tilde x_o}(4R)$. Denote by  $\tilde H_{\tilde \Omega}(x,y,t)$ the  Dirichlet heat kernel on $\tilde \Omega$. Then for all $\delta\in (0,1)$, $x,y\in \tilde B_{\tilde x_o}(R/4)$ and $t\leq R^2/4$, we have
\begin{equation*}
\begin{split}
\tilde H_{\tilde \Omega}(x,y,t)  \leq C_1(\delta, n+q) \, & e^{-\lambda_1(\tilde \Omega)t}\, \tilde V^{-1/2}(x,\sqrt{t})\,\tilde V^{-1/2}(y,\sqrt{t}) \\
&\quad \cdot \exp{ [\,-\frac{\tilde{d}^2(x,y)}{(4+\delta)t} + C_2 (n+q) \, \tilde A (R)\,t \,]},
\end{split}
\end{equation*}
where  $\tilde V(x,r)$ is the volume of $\tilde B_x(r)$, $C_1(\delta, n+q), C_2(n+q)$ are positive constants and
\[
\tilde A (R)=  \bigl[ \frac{1}{R^2} (1 + R\sqrt{K})  + K\bigr]^{1/2}.
\]
\end{thm}

\begin{remark} \label{Rmk24}
We note that the exponential in time term that appears in the heat kernel estimate could be improved. Saloff-Coste shows in ~\cite{SaCo} that under the same assumptions as Theorem  {\rm\ref{thm24}} above, for any $R>1$, $\delta\in (0,1)$, $x,y\in \tilde B_{\tilde x_o}(R/4)$ and $t\leq R^2/4$,
\begin{equation*}
\begin{split}
\tilde H_{\tilde \Omega}(x,y,t)  \leq C_3(\delta, n+q) \, & e^{-\lambda_1(\tilde \Omega)t}\, \tilde V^{-1/2}(x,\sqrt{t})\,\tilde V^{-1/2}(y,\sqrt{t}) \\
&\quad \cdot \exp{ [\,-\frac{\tilde{d}^2(x,y)}{C_4(\delta,n+q)\,t} +  C_5(n+q)\,\sqrt{ K\, t} \,]},
\end{split}
\end{equation*}
where $C_3(\delta, n+q), C_4(\delta,n+q),$  and  $C_5(n+q)$ are positive constants. He obtains this sharper in $t$ estimate for a more general  class of elliptic operators on the manifold, by proving a Harnack inequality similar to the one of Theorem {\rm \ref{T3}}. The exponential term with the first Dirichlet eigenvalue of the Laplacian on $\tilde\Omega$ can also be added using Davies' technique \cite{Dav1}. We will be using this upper bound to obtain the heat kernel estimate of Theorem {\rm\ref{Tb1}}.\end{remark}

Similarly we have a lower bound on the heat kernel (see \cite{SaCo})
\begin{thm} \label{thm25}
Under the same assumptions as Theorem {\rm \ref{thm24}}, for all $\delta\in (0,1)$, $x,y\in \tilde B_{\tilde x_o}(R/4)$ and $t\leq R^2/4$, we have
\begin{equation*}
\begin{split}
\tilde H_{\tilde \Omega}(x,y,t)  \geq C_6(\delta, n+q) \, &  \tilde V^{-1/2}(x,\sqrt{t})\,\tilde V^{-1/2}(y,\sqrt{t}) \\
&\quad \cdot \exp{ [\,-\, \frac{\tilde{d}^2(x,y)}{C_7(\delta,n+q) \,t} - C_8 (n+q) \,K \,t \,]},
\end{split}
\end{equation*}
where $C_6(\delta, n+q), C_7(\delta,n+q),$ and $C_8(n+q)$ are positive constants.
\end{thm}

\section{Riemannian Manifolds of a Special Warped Product Form}

The idea of relating the geometry of a  sequence of manifolds to their collapsing space dates back to Fukaya~\cite{Fuk}. In ~\cite{cc},  Cheeger and Colding studied the eigenvalue convergence for a collapsing sequence of manifolds $\{M_i\}$. In their general case, they need to assume that the  gradient of the $k^{th}$ eigenfunction is bounded above {\it uniformly in $i$} (c.f. equation (7.4) of ~\cite{cc}) {and that the limit manifold has bounded diameter, among other conditions}.

Lott observed in ~\cite{lott} that smooth metric measure spaces are examples of collapsed Gromov-Hausdorff limits of Riemannian manifolds.  In the same article, he considered the product $\tilde M= M\times S^q$ for $q>0$, where $S^q$ is the $q$-dimensional unit sphere, with a family of  warped product metrics
\[
g_\eps=g^M+\eps^2e^{-\frac 2qf}g^{S^q},\,\,\eps>0.
\]
We use $\tilde M_\eps$ to denote the Riemannian manifold with Riemannian metric $g_\eps$. The sequence of warped products $(\tilde M_\eps, g_\eps)$ collapses in the Gromov-Hausdorff sense to $M$ as $\eps\to 0$.

It turns out that we can use such a setting to obtain heat kernel estimates in Bakry-\'Emery geometry
from the corresponding results in Riemannian geometry.\footnote{In a recent article S. Li and X.-D. Li also used a similar warped product over $M$ to prove the W-entropy formula for the fundamental solution of the drifting Laplacian (referred to as the Witten Laplacian) on complete Riemannian manifolds with bounded geometry~\cite{LiLi}. In their case however, it was enough to consider a single warped product space over $M$ and not a collapsing sequence.} In this section, we compute the curvature tensor on these warped products.

We employ the convention of indexing local coframes $\omega_m$ in~$M$
and  $\eta_\mu$ in~$S^q$
by Latin and Greek indices respectively,
the range of them being $m\in\{1,\dots, n\}$
and $\mu\in\{n+1,\dots,n+q\}$, respectively.
Then the local frames $(\omega_1,\dots,\omega_n,\eta_{n+1},\dots,\eta_{n+q})$
are the frames for $M\times S^q$. Capital Latin indices are used for indexing the local frames
in $M\times S^q$, \ie\ $A,B \in\{1,\dots,n+q\}$.\\

We begin by computing the Ricci curvature of $\tilde M_\eps$ (see \cite{DoUn} for further computations of the curvature tensor on multiply warped products). We  note that the components $\widetilde  {\rm Ric}_{ij}$  of the Ricci curvature also follow from O'Neil's formula for a Riemannian submersion.

\begin{prop}\label{prop21} The Ricci curvature $\widetilde{\rm Ric}$ of $\tilde M_\eps$  is given by
\begin{align*}
&\widetilde  {\rm Ric}_{ij}=({\rm Ric^{\it q}_{\it f}})_{ij};
\\
&\widetilde  {\rm Ric}_{i\alpha}=0;\\
&\widetilde  {\rm Ric}_{\alpha\beta}=((q-1)\eps^{-2} e^{\frac 2qf}-q^{-1}e^f\Delta e^{-f})\delta_{\alpha\beta},
\end{align*}
where $\Delta$ is the Laplacian on  the Riemannian manifold  $M$ and $\delta_{\alpha\beta}$ is the Kronecker delta symbol.
\end{prop}

\begin{proof}
Let $\omega_1,\cdots, \omega_n$ be  orthonormal coframes of $M$ and let $\eta_{n+1},\cdots,\eta_{n+q}$ be  orthonormal coframes of $S^q$. Let $\omega_{ij}, \eta_{\alpha\beta}$ be the corresponding connection 1-forms. That is
\begin{align*}
& d\omega_i=\omega_{ij}\wedge \omega_j,\quad \omega_{ij}=-\omega_{ji};\\
&d\eta_\alpha=\eta_{\alpha\beta}\wedge\eta_\beta,\quad \eta_{\alpha\beta}=-\eta_{\beta\alpha}.
\end{align*}
Define
\begin{align*}
& \tilde\omega_i=\omega_i;
\\
&\tilde\omega_\alpha=\eps e^{-\frac1qf}\eta_\alpha.
\end{align*}
Then $(\tilde\omega_1,\cdots,\tilde\omega_{n+q})$ is an orthonormal basis of $\tilde M_\eps$.

Let
\begin{align*}
& \tilde \omega_{ij}=\omega_{ij};\\
&\tilde\omega_{\alpha\beta}=\eta_{\alpha\beta};\\
&\tilde\omega_{i\alpha}=-\tilde\omega_{\alpha i}=-q^{-1}f_i\,\tilde\omega_\alpha.
\end{align*}
Then we can verify that
\[
d\tilde \omega_{A}=\tilde\omega_{AB}\wedge\tilde\omega_B
\]
with $\tilde\omega_{AB}=-\tilde\omega_{BA}$.   The curvature tensor of $M\times S^q$ is defined by
\[
d\tilde\omega_{AB}-\tilde\omega_{AC}\wedge\tilde\omega_{CB}=-\frac 12\tilde R_{ABST}\,\tilde\omega_S\wedge\tilde\omega_T.
\]
By comparing the above equation with the Cartan equations on $M$ and $S^q$, we obtain
\begin{align*}
&\tilde R_{ijst}=R_{ijst};\\
&\tilde R_{A\alpha ij}=0;\\
&\tilde R_{\alpha\beta\gamma\delta}=(\eps^{-2}e^{\frac 2qf}-q^{-2}|\nabla  f|^2)(\delta_{\alpha\gamma}\delta_{\beta\delta}
-\delta_{\alpha\delta}\delta_{\beta\gamma});\\
&\tilde R_{Ai\alpha\beta}=0;\\
&\tilde R_{i\alpha j\beta}=\delta_{\alpha\beta} q^{-1}(f_{ij}-q^{-1}f_if_j).
\end{align*}
The formulas for the Ricci curvature follow.\end{proof}

\begin{remark}
The above proposition allows us to show the following: Suppose that $\textup{Ric}_f^q$ is bounded below on $M$ and consider the ball of radius $R$ at a fixed point in $\tilde M_\eps$, $\tilde B_{\tilde x_o}(R)$. Then, for any $R$  there exists a sufficiently small $\eps>0$ such that the Ricci curvature of $\tilde M_\eps$  on $\tilde B_{\tilde x_o}(R)$ has a lower bound. This explains the need for  the localized theorems in Section  {\rm \ref{s2}}.
\end{remark}

\section{Comparison Theorems  and the  Bottom  of the Spectrum}

The Laplacian on the warped product can be written as
\begin{equation}\label{lap}
\Delta_{\eps}=\Delta_f+\eps^{-2}e^{\frac 2q f}\Delta_{S^q},
\end{equation}
where $\Delta_{\eps}$, $\Delta_{S^q}$ are the Laplacians on $\tilde M_\eps$ and $S^q$, respectively.

Define the bottom  of the Rayleigh quotient of the Laplacian on  $\tilde M_\eps$  by
\[
\lambda_1(\tilde M_\eps)=\inf_{u\in\mathcal C_0^\infty (\tilde M_\eps)} \frac{\int_{\tilde M_\eps} |\nabla_\eps u|^2 }{\int_{\tilde M_\eps} u^2},
\]
where $\nabla_\eps$ is the gradient operator on $\tilde M_\eps$. Similarly, we define the bottom  of the Rayleigh quotient of the drifting Laplacian on the Bakry-\'Emery  manifold $M$ by
\[
\lambda_{1,f}(M)=\inf_{u\in\mathcal C_0^\infty ( M)} \frac{\int_{ M} |\nabla u|^2\, e^{-f} }{\int_{M} u^2\, e^{-f}}.
\]

We prove that

\begin{thm} \label{thm25-2}
Using the notation above, we have
\[
 \lambda_1(\tilde M_\eps)=\lambda_{1,f}(M)
\]
for all $\eps>0$.
\end{thm}

\begin{proof} For any $\delta>0$, there exists a function $u\in\mathcal C_0^\infty(M)$ such that
 \[
 \frac{\int_{ M} |\nabla u|^2\, e^{-f} }{\int_{M} u^2\, e^{-f}}\leq\lambda_{1,f}(M)+\delta.
 \]
 However, if we regard the function $u$ as a function on $\tilde M_\eps$, we have
 \[
 \frac{\int_{\tilde M_\eps} |\nabla_\eps u|^2 }{\int_{\tilde M_\eps} u^2}= \frac{\int_{ M} |\nabla u|^2\, e^{-f} }{\int_{M} u^2\, e^{-f}}\leq\lambda_{1,f}(M)+\delta.
 \]
 By the variation principle we get \[
 \lambda_1(\tilde M_\eps)\leq \lambda_{1,f}(M)+\delta,
 \]
 and hence $ \lambda_1(\tilde M_\eps)\leq \lambda_{1,f}(M)$.

On the other hand, for any $\delta>0$, there exists a  $u=u(x,\xi)\in\mathcal C_0^\infty(M\times S^q)$ such that
\[
 \frac{\int_{\tilde M_\eps} |\nabla_\eps u|^2 }{\int_{\tilde M_\eps} u^2}\leq\lambda_1(\tilde M_\eps)+\delta.
 \]
Let $\n_x$ be the gradient operator with respect to the $x$-component.  For fixed $\xi$, we have
 \[
 \lambda_{1,f}(M)\leq \frac{\int_{ M} |\nabla_x u(x,\xi)|^2\, e^{-f} }{\int_{M} u^2(x,\xi)\, e^{-f}}.
 \]
 Therefore integrating with respect to $\xi$  {over $S^q$, we have}
 \[
  \lambda_{1,f}(M)\leq \frac{\int_{\tilde M_\eps} |\nabla_x u|^2 }{\int_{\tilde M_\eps} u^2}\leq
  \frac{\int_{\tilde M_\eps} |\nabla_\eps u|^2 }{\int_{\tilde M_\eps} u^2}\leq \lambda_1(\tilde M_\eps) +\delta.
  \]
Therefore $  \lambda_{1,f}(M)\leq\lambda_1(\tilde M_\eps)$, and  the theorem is proved.
\end{proof}

\begin{corl} \label{Corl0_T1-new} Let $(M^n,g,e^{-f} dv)$ be a  weighted manifold such that for some positive integer $q$
\[
\textup{Ric}^q_f\geq -K,
\]
where $K$ is a nonnegative number.
Then
\begin{equation*}
\lambda_{1,f}(M) \leq \frac 1{4(n+q-1)}\,K.
\end{equation*}
\end{corl}
\begin{proof}
We fix $x_o$ and $R>2$ and  consider the ball $B_{x_o}(R)$ in $M$. Take  $\tilde x_o\in\tilde M$ with first component $x_o$. For each $R$, we can choose $\eps(R)$ small enough  such that for all $\eps<\eps(R)$
\[
B_{x_o}(4R)\times S^q\subset \tilde B_{\tilde x_o}(4R+1)
\]
where $\tilde B_{\tilde x_o}(r)$ denotes the ball of radius $r$ in $\tilde M_\eps$. Furthermore, we can choose $\eps(R)$ even smaller  such that for all $\eps<\eps(R),\,$ $(q-1)\eps^{-2} e^{\frac 2qf}-q^{-1}e^f\Delta e^{-f}\geq  0$ on $\tilde B_{\tilde x_o}(4R+1)$, since $\tilde B_{\tilde x_o}(4R+1)\subset B_{x_o}(4R+1)\times S^q$.

By Proposition \ref{prop21}, the Ricci curvature of the warped product $\tilde M_\eps$ satisfies
\[
\widetilde  {\rm Ric}\geq -K
\]
on $\tilde B_{\tilde p}(4R+1)$ for all $\eps<\eps(R)$.

Let $\lambda_{1,f}(B_{x_o}(R))$ be the first Dirichlet eigenvalue of $\Delta_f$ on $B_{x_o}(R)$ and let
$\lambda_1(B_{x_o}(R)\times S^{q}, g_\eps)$  be the first Dirichlet eigenvalue of $\Delta_\eps$ on $B_{x_o}(R)\times S^{q}$.

Using the same method as in the proof of Theorem~\ref{thm25-2}, for any $\eps>0$,
\[
\lambda_{1,f}(B_{x_o}(R))=\lambda_1(B_{x_o}(R)\times S^{q}, g_\eps).
\]
By the eigenvalue comparison theorem of Cheng~\cite{SchoenYau_bk}*{Theorem 1, Chapter III.3}  we  have
\[
 \lambda_1(B_p(R)\times S^{q}, g_\eps)\leq \frac 1{4(n+q-1)}\,K+ C(n+q) R^{-2}
 \]
for some constant $C(n+q)$  that only depends on $n+q$, {and for a sufficiently large number $R$}.  Then the upper {bound} estimate follows since $\lambda_{1,f}(B_p(R))\to \lambda_{1,f}(M)$ as $R\to \infty$ and the right side is bounded.
\end{proof}

We note that whenever $\textup{Ric}^{\bar q}_f\geq -K$ for some $\bar q>0$, then $\textup{Ric}^{q}_f\geq -K$ for some positive integer $q$.
We recall that Munteanu and Wang in~\cites{MuW2,MuW} have demonstrated a similar upper bound for $\lambda_1(M)$ under the assumption that  $f$ has linear growth  at a point and  $\textup{Ric}_f$ has a lower bound.

\begin{lem} \label{VolCom}
Let $(M^n,g,e^{-f} dv)$ be a  weighted manifold such that $\textup{Ric}_f^q\geq - K$ for some $K\geq 0$. Then for any $1\leq r<R$
\[
\frac{V_f(x,R)}{V_f(x,r)} \leq \,\frac{1}{r^{n+q}} \, e^{C\,\sqrt{K}\, R+C'}
\]
where $C, C'$  are constants that depend on $n+q$.

\end{lem}
\begin{proof}
Fix $R_o>0$ large, $x_o\in M$ and $\tilde x_o=(x_o,w) \in M\times S^q$.  As before, we use $\tilde V(\cdot,r)$ to denote the volume of a ball of radius $r$ in $\tilde M_\eps$. As in the proof of Corollary \ref{Corl0_T1-new} we take $\eps(R_o)>0$ small enough such that the Ricci curvature of the warped product $M\times S^q$ satisfies
\[
\widetilde{\rm Ric}\geq -K
\]
on $ B_{x_o} (4R_o)\times S^q$ and therefore on $\tilde B_{\tilde x_o} (4R_o)$. Then for any $1\leq r<R<R_o$  and $x \in B_{x_o}(R_o)$
\[
\frac{\tilde V((x,w),R)}{ \tilde V((x,w),r)} \leq  \,\frac{1}{r^{n+q}} \,  e^{C\,\sqrt{K}\, R  +C'}
\]
by the Bishop volume comparison theorem on  {$\tilde M_\eps$.}

It is clear that for any $r>0$, if $\eps$ is  small enough, then there exists a $\delta(\eps)>0$ also small enough such that
\[
\tilde B_{(x,w)}(r)\subset  B_x (r) \times S^q\subset \tilde B_{(x,w)}(r+\delta)
\]
and $\delta(\eps)\to 0$ as $\eps \to 0$.

Thus we have
\begin{align*}
& \eps^q\, V_f(x,R)\leq\tilde V((x,w), R+\delta);\\
&\eps^q\, V_f(x,r)\geq \tilde V((x,w), r).
\end{align*}
The result follows by letting $\eps\to 0$.
\end{proof}

We fix a point $x_o$ in $M$ and define $r(x)=d(x,x_o)$ to be the radial distance to $x_o$. In the following theorem we see that the natural assumption for a Laplacian comparison theorem on $M$ is that the $q$-Bakry-\'Emery Ricci tensor is bounded below.
\begin{thm}\label{L0-new}
Let $(M^n,g,e^{-f} dv)$ be a weighted manifold such that
\[
{\rm Ric}_f^q\geq -K
\]
for some positive constant $q$ and $K\geq0$.
Then,
\[
\limsup_{r\to\infty} \Delta_f r\leq \sqrt{(n+q)K}
\]
in the sense of distribution. Moreover,
\[
\Delta_f r(x)\leq (n+q)\;\frac{1}{r(x)} + \sqrt{(n+q)K}
\]
for $x\neq x_o$ in the sense of distribution.
\end{thm}

\begin{proof} We use the Bochner formula~\eqref{fBochner} for the distance function $r$:
\begin{equation*}
\begin{split}
0=\frac 12\Delta_f |\n r|^2&=|\nabla^2 r|^2+\< \n r , \n (\Delta_f r)\>+{\rm Ric}_f (\n r,\n r)\\
&=|\nabla^2 r|^2+\frac{ \p (\Delta_f r)}{\p r}+{\rm Ric}_f^q (\n r,\n r) +\frac{1}{q}  \<\n f,\n r\>^2\\
&\geq \frac{1}{n}(\Delta r)^2  +\frac{ \p (\Delta_f r)}{\p r}  -K + \frac 1q \<\n f,\n r\>^2,
\end{split}
\end{equation*}
where we have applied the inequality  $|\n^2 r|^2\geq \frac{1}{n}(\Delta r)^2$, the assumption on ${\rm Ric}_f^q$ and the fact that $|\n r|=1$ a.e. on $M$.

For any positive integers $n,q$, the following inequality holds
\[
\frac 1n x^2+\frac 1q y^2\geq\frac{1}{n+q}(x+y)^2.
\]
Therefore,
\begin{equation}\label{abcd}
\frac{1}{n}(\Delta r)^2+\frac 1q \<\n f,\n r\>^2\geq\frac{1}{n+q} (\Delta_f r)^2.
\end{equation}
Combining this lower bound with the Bochner formula, we get
\[
\frac{1}{n+q} (\Delta_f r)^2+\frac{\pa(\Delta_f r)}{\pa r}-K\leq 0.
\]
By the Riccati equation comparison, {for $K>0$} we obtain
\[
\Delta_f r \leq \sqrt{(n+q)K}\; \frac{\cosh (\alpha r)}{\sinh (\alpha r)}
\]
in the sense of distribution, where $\alpha=\sqrt{K/(n+q)}$, {whereas for $K=0$
\[
\Delta_f r \leq \frac{n+q}{r}
\]
}.
\end{proof}

The above theorem is {known the literature (see for example~\cite{BQ}*{Theorem 4.2} and ~\cite{XDLi1}). It is} very similar to Wei-Wylie \cite{WW} and Munteanu-Wang~\cite{MuW}, but  the lower bound on the $q$-Bakry \'Emery Ricci tensor allows as to make no assumption on the linear or sublinear growth of $f$. Wei-Wylie~\cite{WW} also proved a volume comparison theorem using the above drifting Laplacian comparison theorem.

As a corollary to the {above} Laplacian comparison theorem, we prove the following Barky-\'Emery version of the uniqueness of the heat kernel.
\begin{corl}\label{cor42}  Under the assumptions of the above theorem, the heat kernel $H_f(x,y,t)$ of the Laplacian $\Delta_f$ on $M$ is unique.
\end{corl}

\begin{proof}
The proof is similar to that in the Riemannian case, which was done by Dodziuk~\cite{dod}.
All we need is the following version of the maximum principle. Consider the Cauchy problem
\[
\left\{
\begin{array}{ll}
{\displaystyle \frac{\pa u}{\pa t}-\Delta_f u=0 }& \text{on } M\times (0,T)\\
& \\
u(x,o)=u_0(x) &\text{on } M.
\end{array}
\right.
\]
Then every bounded solution is uniquely determined by the initial data.

To prove the above claim, we take the following function
\[
v=u-C_1-\frac{C_2}{R}(r(x)+C_3 t),
\]
where $C_1,C_2, C_3$ are positive constants to be determined. On the set $\pa B_{x_0}(R)\times (0,T)$, for $C_1$ sufficiently large, $v\leq 0$. On the other hand, in the sense of distribution, we have
\[
\frac{\pa v}{\pa t}\leq\Delta_f v.
\]
Thus by the maximum principle, we have
\[
v(x,t)\leq v(x,0).
\]
Letting $R\to \infty$, we obtain
\[
u(x,t)\leq \sup |u(x,0)|.
\]
Replacing $u$ by $-u$, we obtain
\[
-u(x,t)\leq \sup |u(x,0)|.
\]
The claim and hence the corollary is proved.

We end this section by a slightly different version  of the Laplacian comparison theorem  from  Theorem~\ref{L0-new}, which will be used in Section~\ref{sec6}.  We make the following definition

\begin{Def}\label{defnonnegative}
We say that ${\rm Ric}_f^q$ is asymptotically nonnegative in the radial direction,
if there exists a continuous function $\delta(r)$ on $\mathbb{R}^+$ such that
\begin{enumerate}
\item $\lim_{r\to\infty} \delta(r)=0$

\item $\delta(r)>0$  and

\item  For some $q>0$, ${\rm Ric}_f^q(\n r, \n r)\geq -\delta(r)$.
\end{enumerate}
\end{Def}

Following the proof of Theorem \ref{L0-new} above and Lemma 4.2 in~\cite{char-lu-2}, we get the Laplacian comparison result

\begin{lem} \label{L1}
Let $(M^n,g,e^{-f} dv)$ be a  weighted manifold.
Assume that for some $q>0$, ${\rm Ric}_f^q$ is asymptotically nonnegative in the radial direction. Then
\begin{equation} \label{DeltaEst}
\limsup_{r\to\infty} \Delta_f r= 0
\end{equation}
in the sense of distribution.
\end{lem}
\end{proof}

\section{The Heat Kernel in  Bakry-\'Emery Geometry}

\subsection{Comparing the heat kernels on Riemannian and Bakry-\'Emery manifolds}

On any complete noncompact Riemannian manifold, the heat kernel exists and is positive \cite{SchoenYau_bk}*{Theorem 1, Chapter III.2}. Thoughout this paper, we study the Friedrichs extension of the Laplacian. The heat kernel corresponding to this extension is the smallest positive heat kernel.

We make similar definition in the Bakry-\'Emery case. Let $H_f(x,y,t)$ be the heat kernel of $\Delta_f$ corresponding to the Frederichs extension. Then it is the smallest positive heat kernel among all other heat kernels that correspond to heat semi-groups of   self-adjoint extensions of $\Delta_f$.

The idea is to obtain estimates for $H_f(x,y,t)$ by comparing the heat kernel of $\Delta_f$ on $M$ to averages of the heat kernel of the Laplacian $\Delta_\eps$ on the warped product considered in the previous  sections.

We use $\tilde H_\eps((x,\omega), (y,\xi),t)$ to denote the heat kernel of $\Delta_\eps$ on the warped product $\tilde M_\eps$. Define
\[
H_\eps((x,\omega), y,t)=\eps^q\int_{S^q(1)} \tilde  H_\eps ((x,\omega), (y,\xi),t)d\xi,
\]
where the integration is over the standard metric on $S^q(1)$.  We specifically use $S^q(1)$ in place of $S^q$ to emphasize that the metric is the standard metric with constant curvature $1$. We have the following lemma:

\begin{lem} \label{Lb2}
 The function
$H_\eps((x,\omega), y,t)$ is independent of $\omega$.
\end{lem}

\begin{proof}
Let $G$ be the isometry group of $S^q(1)$ preserving the orientation. Let $\omega_1,\omega_2\in S^q$, and  let $A\in G$ such that $A\omega_1=\omega_2$. Then we have
\begin{align*}
&H_\eps((x,\omega_2), y,t)=H_\eps((x,A\omega_1), y,t)\\
&=\eps^q\int_{S^q(1)} \tilde  H_\eps ((x,A\omega_1), (y,\xi),t)d\xi\\
&=\eps^q\int_{S^q(1)} \tilde  H_\eps ((x,A\omega_1), (y,A\xi),t)d(A\xi)\\
&=\eps^q\int_{S^q(1)} \tilde  H_\eps ((x,A\omega_1), (y,A\xi),t)d\xi\\
&=\eps^q\int_{S^q(1)} \tilde  H_\eps ((x,\omega_1), (y,\xi),t)d\xi= H_\eps((x,\omega_1), y,t).
\end{align*}
Thus $H_\eps((x,\omega), y,t)$ is independent of $\omega$.
\end{proof}

In what follows, we write
\[
H_\eps(x,y,t)= H_\eps((x,\omega), y, t).
\]
It can be regarded either as a function on $M$ or as a function on $\tilde M_\eps$.

\begin{lem} \label{Lb1}
\[
\lim_{t\to 0} H_\eps(x,y,t)=\delta_x(y) e^{f(x)}
\]
with respect to the weighted measure on $M$.
\end{lem}

\begin{proof}
Let $\varphi(y)$ be a smooth function on $M$ with compact support. Then
\begin{align*}
&\lim_{t\to 0}\int_M  H_\eps(x,y,t)\varphi(y) \, {e^{-f(x)}} dy\\
&=\lim_{t\to 0}\,\eps^q \int_{M\times S^q(1)}  \tilde  H_\eps ((x,\omega), (y,\xi),t)\varphi(y)\, {e^{-f(x)}} dy\,d\xi\\
&=\lim_{t\to 0}\,\int_{\tilde M_\eps}  \tilde  H_\eps ((x,\omega), (y,\xi),t)\varphi(y) dv\\
&=\varphi(x).
\end{align*}
The lemma is proved.
\end{proof}

Let $\Omega$ be compact domain of $M$ with smooth boundary and let $\tilde \Omega=\Omega\times S^q$. Let $\tilde H_{\eps,\tilde \Omega}((x,\omega),(y,\xi),t)$ be the Dirichlet  ({or} Neumann)  heat kernel of the Riemannian Laplacian $\Delta_\eps$ on $\tilde \Omega = \Omega\times S^q$ with respect to the $g_\eps$ metric and let $H_{f,\Omega}$ be the Dirichlet  ({resp.} Neumann)  drifting   heat kernel on $\Omega$.
\begin{corl}\label{cor53} Using the above notation, we have
\[
H_{f, \Omega}(x,y,t)=\eps^q\int_{S^q(1)}\tilde H_{\eps,\tilde\Omega}((x,\omega),(y,\xi),t) d\xi.
\]
\end{corl}
\begin{proof}
Let
\[
H_{\eps,\Omega}(x,y,t)=\eps^q\int_{S^q(1)}\tilde H_{\eps,\tilde\Omega}((x,\omega),(y,\xi),t) d\xi.
\]
Obviously,
$H_{\eps,\Omega}$ satisfies the heat kernel equation
\[
\left(\frac{\pa}{\pa t}-\Delta_{f,x}\right) H_{\eps,\Omega}(x,y,t)=0
\]
and $ H_{\eps,\Omega}(x,y,t)=  H_{\eps,\Omega} (y,x,t)$.
Therefore both $H_{f,\Omega}$ and $H_{\eps,\Omega}$ satisfy the heat kernel equation for the drifting Laplacian with the same initial  value (cf. Lemma~\ref{Lb1})  and boundary value. The corollary follows from the maximum principle.
\end{proof}

The following corollary generalizes~\cite{Lu-Row}*{Corollary 1}, where only the Neumann eigenvalues were considered.

\begin{corl}
Let $\lambda_{k,\eps}(\tilde \Omega)$ be the Dirichlet  ({resp.} Neumann) eigenvalues of {$\Delta_\eps$ on}  $\tilde\Omega$ and let $\lambda_{k,f}(\Omega)$ be the Dirichlet  ({resp.} Neumann) eigenvalues of the drifting  Laplacian on $\Omega$. Then
\[
\lambda_{k,\eps}(\tilde \Omega)\to\lambda_{k,f}(\Omega)
\]
for $\eps\to 0$.

\end{corl}
\begin{proof}
On a compact manifold,  $H_{f, \Omega}(x,y,t)$ has the eigenfunction expansion
\[
H_{f, \Omega}(x,y,t)= \sum_{k=1}^\infty e^{-\lambda_{k,f}t} \phi_{k,f}(x)\, \phi_{k,f}(y)
\]
where the $\phi_{k,f} $ are eigenfunctions forming an orthonormal basis of the space of weighted $L^2$ integrable functions of $\Omega$, such that $\phi_{k,f} $ corresponds to the eigenvalue $\lambda_{k,f}$. Similarly,  on $\tilde\Omega$
\[
\tilde H_{\eps,\tilde\Omega}((x,\omega),(y,\xi),t) =\sum_{k=1}^\infty e^{-\lambda_{k,\eps} t} \phi_{k,\eps} (x,\omega)\, \phi_{k,\eps}(y,\xi)
\]
where the  $\phi_{k,eps} $ are eigenfunctions forming an orthonormal basis of the space of $L^2$ integrable functions of $\tilde\Omega$, such that $\phi_{k,\eps} $ corresponds to the eigenvalue $\lambda_{k,\eps}$.  Therefore,
\[
\int_{\Omega} H_{f, \Omega}(x,x,t) e^{-f(x)} dx = \sum_{k=1}^\infty e^{-\lambda_{k,f} t}
\]
whereas,
\[
\int_{\tilde \Omega} H_{\eps,\tilde\Omega}((x,\omega),(x,\omega),t)\;  \eps^q \,e^{-f(x)} d\omega \,dx = \sum_{k=1}^\infty e^{-\lambda_{k,\eps} t}.
\]
By Corollary~\ref{cor53}, we have
\[
\sum_{k=1}^\infty e^{-\lambda_{k,f}t}=\sum_{k=1}^\infty e^{-\lambda_{k,\eps}t}
\]
for any $t>0$. Thus the conclusion of the corollary follows since the functions $e^{-\lambda t}$ are linearly independent for distinct $\lambda$.
\end{proof}

On a noncompact complete manifold, the maximum principle does not apply directly. Moreover,  the heat kernel might not be unique in general.\footnote{By the result of Dodziuk~\cite{dod} and Corollary~\ref{cor42}, the heat kernels are {indeed} unique under the assumption that the tensors $\widetilde{{\rm Ric}}$ on $\tilde M_\eps$ and ${\rm Ric}_f^q$ on $M$ are bounded below.}
Nevertheless, we still have the following
\begin{thm}\label{Tb2} For $\eps>0$, we have
\[
H_f(x,y,t)=H_\eps(x,y,t).
\]
\end{thm}

\begin{proof}
Let $\{\Omega_i\}$ be an exhaustion of $M$ by bounded domains. That is
\begin{enumerate}
\item $\Omega_i$ are bounded domains in $M$ with smooth boundary,
\item $\Omega_i\subset\Omega_{i+1}$ for any positive integer $i$ and
\item $\bigcup\Omega_i=M$.
\end{enumerate}
Then we have
\begin{align*}
&
\lim_{i\to\infty} H_{f,\Omega_i}(x,y,t)=H_f(x,y,t)\\
&\lim_{i\to\infty} \tilde H_{\eps,\tilde \Omega_i}((x,\xi),(y,\eta),t)=\tilde H_\eps((x,\xi),(y,\eta),t),
\end{align*}
where $\tilde\Omega_i=\Omega_i\times S^q$. The theorem follows from Corollary~\ref{cor53}.
\end{proof}

\begin{thm} \label{Tb1}
Let $(M^n,g,e^{-f} dv)$ be a  weighted manifold  such that for some positive integer $q$,
\[
\textup{Ric}^q_f\geq -K
\]
on $B_{x_o}(4R+4)\subset M$. Then for any  $x,y\in B_{x_o}(R/4)$, $t<R^2/4$ and $\delta\in(0,1)$
\begin{equation}
\begin{split}\label{heat41-new}
H_{f}(x,y,t) \leq  \;C_3(\delta, n+q)  & \;   V_f^{-1/2}(x,\sqrt{t}) \,V_f^{-1/2}(y,\sqrt{t}) \\
&\quad   \cdot  \exp [\,-\lambda_{1,f}(M)\,t  -\frac{d^2(x,y)}{C_4(\delta, n+ q)\,t} + C_5(n+q) \sqrt{K\,t} \,]
\end{split}
\end{equation}
for some positive constants $C_3(\delta,n+q),  C_4(\delta,n+q)$ and  $C_5(n+q).$

Whenever $\textup{Ric}^q_f\geq -K$ on $M$, then the same bound also holds for all $x,y \in M$ and $t>0$.
\end{thm}

\begin{proof} Let $\Omega$ be a compact domain in $M$ large enough such that $B_{x_o}(4R+4)\subset \Omega$. Let $\tilde x_o$ be a point in $\tilde M$ with first component $x_o$ and $\tilde \Omega= \Omega \times S^q$. Since $\tilde B_{\tilde x_o}(4R+4) \subset B_{x_o}(4R+4) \times S^q$, as in the proof of Corollary \ref{Corl0_T1-new}, we take $\eps(R)$ small enough such that for all $\eps<\eps(R)$ the Ricci curvature of $\tilde M_\eps$ satisfies
\[
\widetilde  {\rm Ric}\geq -K
\]
on $B_{x_o}(4R+4)\times S^q$ and therefore on $\tilde B_{\tilde x_o}(4R+4)\subset \tilde\Omega$. We observe that for any  $x, y\in B_{x_o} (R/4)$, the points $(x,\omega), (y,\xi) \in \tilde B_{\tilde x_o} ((R+1)/4)$ for $\eps$ small enough.

Therefore, by Theorem~\ref{thm24}, Remark~\ref{Rmk24} and Theorem \ref{Tb2}, for any $x, y\in B_{x_o} (R/4)$, $\delta\in(0,1)$ and $t\leq R^2/4$, we have \begin{equation*}
\begin{split}
H_{f}(x,y,t) & \leq \eps^q \, C_3(\delta, n+q)  \,\exp[ -\lambda_1(\tilde \Omega)t +  C_5(n+q) \sqrt{K\,t} ]\\
& \quad \cdot \, \int_{S^q(1)} \tilde V^{-1/2}((x,\omega),\sqrt{t})\,\tilde V^{-1/2}((y,\xi),\sqrt{t}) \cdot \exp[ -\frac{\tilde{d}^{\,2}((x,\omega),(y,\xi))}{C_4(\delta, n+ q)\,t}  ]\, d\xi
\end{split}
\end{equation*}
where $\tilde{d}(\cdot,\cdot)$ is the distance function in  $\tilde M$ and $\tilde V(\tilde{x},r)$ is the volume  of the ball at $\tilde x$ with radius $r$ contained in $\tilde M_\eps$.

Given that $\tilde{d}((x,\xi),(y,\eta))\geq d(x,y)$, and using Corollary \ref{Corl0_T1-new}  we get \begin{equation*}
\begin{split}
H_{f}(x,y,t)\leq &\, C(\delta, n,q)  \, \exp [-\lambda_{1,f}(\Omega)t \,+ C_5(n+q) \sqrt{K\,t}\,  -\frac{{d}^2(x,y)}{C_4(\delta, n+ q)\,t} ]\, \\
&\, \cdot \eps^q\, \int_{S^q(1)} \tilde V^{-1/2}((x, \omega),\sqrt{t})\,\tilde V^{-1/2}((y,\xi),\sqrt{t}) \, d\xi.
\end{split}
\end{equation*}

For any fixed $(x,\omega)$ and $t>0$, there exists an $\eps_o(x,t)$ small enough such that for all $\eps<\eps_o$
\[
\tilde B_{(x,\omega)}(\sqrt t) \supset B_x(\sqrt t-C\eps)\times S^q
\]
for a constant $C$ which may depend on $x$. It follows that
\[
\tilde V((x,\omega),\sqrt t)\geq \eps^q\,V_f(x,\sqrt t-C\eps)
\]
for $\eps$ sufficiently small.

The theorem is proved  by sending $\eps\to 0$ in the right  side.
\end{proof}

\begin{remark}
We would like to further point out that only assuming $Ric_f$ bounded below and $f$ of linear growth at a point as in \cite{MuW2} would not be sufficient to obtain the global heat kernel  bound this way. In order to get the heat kernel estimates,  an assumption on the {\it uniform} linear growth of $f$ is needed, which is almost equivalent to assuming that   the gradient of $f$ is bounded.
\end{remark}

Theorem \ref{Tb2} also shows that lower bounds on the heat kernels   $\tilde H_{\eps, \tilde \Omega_i}$ imply   lower bounds on $H_{f, \Omega_i}$. Theorem \ref{thm25} of Saloff-Coste gives us the following lower estimate:
\begin{thm}
Let $(M^n,g,e^{-f} dv)$ be a  weighted manifold such that for some positive integer $q$ and $K\geq0$,
\[
\textup{Ric}^q_f\geq -K
\]
on $B_{x_o}(4R+4)\subset M$. Then for any  $x,y\in B_{x_o}(R/4)$, $t<R^2/4$ and $\delta\in(0,1)$
\begin{equation*}
\begin{split}
  H_{\Omega}(x,y,t)  \geq C_6(\delta, n+q) \, &    V^{-1/2}(x,\sqrt{t})\,  V^{-1/2}(y,\sqrt{t}) \\
&\quad \cdot \exp{ [\,-C_7(\delta,n+q)\, \frac{ {d}^2(x,y)}{t} - C_8 (n+q) \,K \,t \,]}
\end{split}
\end{equation*}
for some  positive constants $C_6(\delta, n+q), C_7(\delta, n+q)$ and $C_8(n+q)$.
Whenever $\textup{Ric}^q_f\geq -K$ on $M$, then the same bound also holds for all $x,y \in M$ and $t>0$.
\end{thm}

\subsection{Other inequalities}\label{sub51}

Although it will not be used in this paper, we would like to pursue the Li-Yau type of inequality in Barky-\'Emery geometry. Using Theorem~\ref{r1}, we obtain

\begin{thm}
Let $(M^n,g,e^{-f} dv)$ be a weighted manifold such that for some positive integer $q$ and $K\geq 0$
\[
\textup{Ric}^q_f\geq -K.
\]
Suppose that $u$ is a positive solution of the equation
\[
(\frac{\p}{\p t} -\Delta_f)u=0
\]
on $M\times[0,T]$. Then for any $\alpha>1$, $u$ satisfies the Li-Yau type estimate
\[
\frac{|\n u|^2}{u^2}-\alpha \frac{u_t}{u} \leq  \frac{(n+q) \alpha^2}{2t} +    \frac{(n+q) \, K \, \alpha^2}{2(\alpha-1)} +C(n+q)\frac{\alpha^2}{R^2}\left(\frac{\alpha^2}{\alpha-1}+R\sqrt K\right)
\]
for any $x,y\in M$ and $t>0$.
\end{thm}

This kind of estimate was obtained by Qian \cite{Qi} using the Li-Yau technique {for the case $\textup{Ric}_f^q\geq 0$ and X.-D. Li for the case  $\textup{Ric}_f^q\geq -K$~\cite{XDLi1}} (see also~\cite{XDLi2} {and ~\cite{XDLi3} where they are used to obtain upper and lower bound estimates on the heat kernel of the drifting Laplacian}). Here on the other hand, we obtain it effortnessly, by simply considering $u$ as a solution to the heat equation on a large enough ball of some $\tilde M_\eps$ and applying Theorem \ref{r1} to $u$.

\begin{remark}\label{rmk510} If we make the stronger assumptions $\textup{Ric}_f>-K$ and $|\n f|\leq a_o$, where $a_o$ is a constant, then it is possible to prove the following Li-Yau type of inequality:
\[
\frac{|\n u|^2}{u^2}-\alpha \frac{u_t}{u} \leq  \frac{n \alpha^2}{2t} +    \frac{C(n) \, (K+a_o^2) \, \alpha^2}{2(\alpha-1)} +C(n)\frac{\alpha^2}{R^2}\left(\frac{\alpha^2}{\alpha-1}+R(\sqrt K+ a_o)\right).
\]
This estimate is sharper when $t\to 0$, and can be used to study the very initial behavior of the Barky-\'Emery heat kernel (cf. ~\cite{xu}). The proof can be done by following the same procedure as in the classical case (for an outline of the classical argument see ~\cites{Davies,SchoenYau_bk}). The key inequality we use is the one similar to ~\eqref{abcd}. In particular, for any $u$ on $M$ we have
\[
|\n^2 u|^2\geq \frac{1}{n} (\Delta u)^2 \geq  \frac{1}{n} (\Delta_f u)^2 -2 a_o |\Delta_f u|\; |\n u|.
\]
\end{remark}

\section{The $L^p$ Essential Spectrum of the Drifting Laplacian}\label{sec7}
For the last two sections of this paper, we turn to the Bakry-\'Emery essential spectrum. Let $(M^n,g,e^{-f} dv)$ be a weighted manifold. We denote the $L^2$ norm associated to the measure $e^{-f}dv$ by $$\| u\|_2=\int_M |u|^2 \, e^{-f}dv.$$  Let $$L^2(M,e^{-f}dv)=\{u\mid \|u\|_2<\infty\}$$ be the real Hilbert space.
Similar to the $L^2$ norm, we can define the $L^p$ norms $$\| u\|_p=\int_M |u|^p \, e^{-f}dv,$$  and the Babach spaces $$L^p(M,e^{-f}dv)=\{u\mid \|u\|_p<\infty\}.$$
Since $-\Delta_f$ is  nonnegative definite, the $L^2$ essential spectrum is contained in the nonnegative real line.

\begin{lem}  \label{LE1}
The heat semigroup $e^{t\Delta_f}$  is well-defined  on $L^p$ for all $p\in [1,\infty]$. This in particular implies that the operator $e^{t\Delta_f}$ is bounded on $L^p$ for any $t\geq 0$.
\end{lem}

This is a direct consequence of Kato's inequality
\[
\int_M \<\n |u|, \n |u| \> \, e^{-f}dv\leq \int_M \<\n u, \n u \> \, e^{-f}dv.
\]
 (see Davies~\cite{Davies}*{Theorems 1.3.2, 1.3.3})

We denote by $\Delta_{p,f}$ the infinitesimal generator of $e^{t\Delta_f}$ in the $L^p$ space.
We  denote the essential spectrum of $-\Delta_{p,f}$ by $\sigma(-\Delta_{p,f})$ and its resolvent set by $\rho(-\Delta_{p,f})=\mathbb{C}\setminus \sigma(-\Delta_{p,f})$. We define $\Delta_{\infty,f}$ to be the dual operator of $L^1$ so its spectrum is identical to that of $\Delta_{1,f}$.

We will find sufficient conditions on the manifold such that the essential spectrum of $\Delta_{p,f}$ is independent of $p$. Similar to the case of the Laplacian on functions and  differential  forms it will be necessary to assume that the volume of the manifold, with respect to the weighted measure neither decays nor grows exponentially~\cites{sturm,Char1}. Let  $V_f(x,r)$ be  the volume of the geodesic ball of radius $r$ at $x$ with respect to the $e^{-f}dv$ measure. Then
\begin{Def}
We say that the volume of a weighted manifold  $(M^n, g, e^{-f} dv)$ grows uniformly subexponentially, if for any $\eps>0$ there exists a constant $C{(\eps)}$, independent of $x$, such that for any $r>0$ and $x\in M$
\[
V_f(x,r)\leq C{(\eps)} \, V_f(x,1) \, e^{\eps \, r }.
\]
\end{Def}

Define $\phi(x)=V_f(x,1)^{-1/2}$. The assumption on the uniformly subexponential volume growth allows us to prove the following estimate
\begin{lem} \label{VolCom2}
Let $(M,g,e^{-f} dv)$ be a weighted manifold whose volume grows uniformly subexponentially. Then for any $\beta >0$
\[
\sup_{x\in M}  \int_M \phi(x) \, \phi(y) \, e^{-\beta \, d(x,y)} \, e^{-f(y)}dv(y) < \infty.
\]
\end{lem}

\begin{proof}
The uniformly subexponential volume growth of $(M, g, e^{-f}dv)$ shows that for each $\eps>0$ there exists a finite constant $C(\eps)$ such that
\[
V_f(x,1)\leq V_f(y, 1+d(x,y))\leq C(\eps) \, V_f(y,1) \, e^{\eps(1+d(x,y))}.
\]
Therefore
\[
\phi(y)^2 \leq \phi(x)^2\, C(\eps)\,  e^{ \eps(1+d(x,y))}.
\]
We get that
\begin{equation*}
\begin{split}
&\sup_{x\in M}  \int_M \phi(x) \, \phi(y) \, e^{-\beta \, d(x,y)} \,  e^{-f(y)}dv(y)\\ & \leq   C'(\eps)\sup_{x\in M}   \, \phi(x)^2  \int_M  e^{-\beta \, d(x,y)} \, e^{\frac 12\eps(1+d(x,y))}\,e^{-f(y)}dv(y)\\
& \leq  C'(\eps) \sup_{x\in M}   \, \phi(x)^2  \sum_{j=0}^{\infty}  e^{-\beta \,j}\,e^{\frac 12 \eps (j+2)} V_f(x,j+1)\\
&\leq   C'(\eps)\sup_{x\in M}    \, \sum_{j=0}^{\infty} C(\eps) \, e^{-\frac 12\beta \,j+ \frac 32 \eps(j+1)} < \infty
\end{split}
\end{equation*}
after choosing $\eps<\frac 14\beta$.
\end{proof}

We will prove the following theorem

\begin{thm} \label{ThmLp}
Let $(M^n,g,e^{-f} dv)$ be a weighted manifold whose volume grows uniformly subexponentially with respect to the measure $e^{-f} dv$. Assume that the $q$-Bakry-\'Emery Ricci curvature  satisfies  $\textup{Ric}_f^q\geq - K$  for some $q>0$ and $K\geq0$. Then the $L^p$ essential spectrum of the drifting Laplacian is independent of $p$  for all $p\in[1,\infty]$.
\end{thm}

Our proof will follow closely to  that of Sturm~\cite{sturm}, which was initially developed by Hempel and Voigt for Schr\"odinger operators on $\mathbb{R}^n$ \cite{HemVo}.

Define the set of functions
\[
\Psi_{\eps}= \{ \psi \in {\mathcal C}_0^{1}(M) \, | \, |\n \psi|\leq \eps  \}.
\]
Observe that  for any pair of points $x, y \in M$, $\sup\{\psi(x)-\psi(y)\, | \, \psi \in  \Psi_{\eps}\} \leq \eps \, d(x,y).$ We will be considering perturbations of the operator $\Delta_{2,f}$ of the form $e^{\psi}\Delta_{2,f} e^{-\psi}$. The following lemma is technical and is straightforward from Kato's result~\cite{Kato}*{Theorem VI.3.9}.

\begin{lem} \label{LemKato}
For any compact subset $W$ of the resolvent set of $\rho(-\Delta_{2,f})$ there exist $\eps>0$ and $C<\infty$ such for all $\xi \in W$ and $\psi\in \Psi_{\eps}$, $\xi$ belongs to the resolvent set of the operator $-e^{\psi}\Delta_{2,f} e^{-\psi}$ and
\[
\|(-e^{\psi}\Delta_{2,f} e^{-\psi}-\xi)^{-1}\|_{L^2\to L^2} \leq C.
\]
\end{lem}

The geometric conditions on the manifold stem from the necessity to show the following proposition
\begin{prop} \label{ResProp}
Let $(M^n,g,e^{-f} dv)$ be a weighted manifold such that $\textup{Ric}_f^q\geq - K$ for some $q>0$ and $K\geq 0$.
Then for any $\xi\in \rho(-\Delta_{2,f})$, there exists   a positive integer $N$ such that for any $ N\leq m \leq 2N$,  $(-\Delta_{2,f} -\xi)^{-m}$ has a smooth kernel function $G_\xi(x,y)$ satisfying
\[
|G_\xi(x,y)|\leq C\phi(x)\phi(y) e^{-\eps d(x,y)}
\]
for some $\eps>0$ and $C>1$.

Under the additional assumption that the volume of $M$ grows uniformly subexponentially with respect to the weighted volume $e^{-f} dv$, the operator $(-\Delta_{2,f} -\xi)^{-m}$ is bounded from $L^p$ to $L^p$ for all $1\leq p \leq \infty$.
\end{prop}

The proof of Proposition \ref{ResProp} requires the following  upper bound estimate for the heat kernel of the drifting Laplacian, which follows from Lemma \ref{VolCom} and Theorem \ref{Tb1}
\begin{corl} \label{CorlT2b}
Let $(M^n,g,e^{-f} dv)$ be a weighted manifold such that   $\textup{Ric}_f^q\geq -K$ for some positive integer $q$ and $K\geq 0$. Then for any $\beta_1 >0$ there exists a negative number $\alpha$ and constant $C(n+q, K,\beta_1)$ such that
\begin{equation} \label{heat42}
H_f(x,y,t)  \leq  C \, \phi(x)^2\, \sup\{t^{-(n+q)/2},1\} \,e^{-\beta_1 \, d(x,y)}\,
 e^{- (\alpha +1) t}.
\end{equation}
\end{corl}

\begin{proof}
For $r<1$, Lemma \ref{VolCom} gives
\[
V_f(x,r)^{-1} \leq \phi(x)^2 \, r^{-(n+q)} \, C(n+q, K),
\]
whereas
\[
V_f(x,r)^{-1} \leq \phi(x)^2
\]
holds for any $r\geq1$. Lemma~\ref{VolCom} also yields
\[
\phi(y)^2 \leq C\,\phi(x)^2\,  e^{\bar C\, \sqrt{K}\,(1+d(x,y))}
\]
for all $x, y \in M$ and constants $C,\bar C$ that only depend on $n+q$.  At the same time the exponential function satisfies the elementary inequality
\[
e^{-d^2/4ct} \leq e^{-\beta_2 \, d} \, e^{c\, \beta_2^2 t}
\]
for any $\beta_2 \in \mathbb{R}$. Combining the above estimates with \eref{heat41-new} we get
\begin{equation*}
\begin{split}
&H_f(x,y,t)  \leq \;  C \, \phi(x)^2\, \sup\{t^{-(n+q)/2},1\} \; \\
&\cdot \, \exp[ -\lambda_{1,f}(M) \, t +\bar C\, \sqrt{K}(1+d(x,y)) - \beta_2 d(x,y)  + C_5 \sqrt{K \,t}  + \frac 14 \beta_2^2 C_4 \,t \,]\cdot{e^{-\alpha t}}.
\end{split}
\end{equation*}
Choosing $\beta_2$ such that $\beta_1= - \bar C\, \sqrt{K} + \beta_2$ is any positive number, we get the result with $\alpha=-1 -C_5 \sqrt{K} - \frac 14 \beta_2^2 C_4 .$
\end{proof}

We are now ready to prove  Proposition \ref{ResProp}.
\begin{proof}[Proof of Proposition {\rm\ref{ResProp}}]
Without loss of generality we take $q$ to be a positive integer.  We note that $(-\Delta_{2,f} -\xi)^{-m}$ has an integral kernel whenever it is bounded from $L^1$ to $L^\infty$. To prove the upper bound for the integral kernel in the proposition, it suffices to show that for some positive integer $N$, whenever $N\leq m\leq 2N$ and $\eps>0$,  the perturbed operator $\phi^{-1} \, e^{\psi} (-\Delta_{2,f} -\xi)^{-m} e^{-\psi} \phi^{-1}$ is bounded from $L^1$ to $L^{\infty}$ for all $\psi\in \Psi_\eps$.

We use the resolvent equation to rewrite this perturbed operator  as
\begin{equation}
\begin{split}
\phi^{-1} \, e^{\psi} (- & \Delta_{2,f} -\xi)^{-m} e^{-\psi} \phi^{-1}= \\
&=
\sum_{j=1}^m \left(
\begin{array}{c}
m \\
j \\
\end{array}
\right)
(\xi-\alpha)^j[\phi^{-1} \, e^{\psi} (-\Delta_{2,f} -\alpha)^{-m/2} e^{-\psi}]\\
& \qquad \cdot [e^{\psi} (-\Delta_{2,f} -\xi)^{-1} e^{-\psi}] \; [ e^{\psi} (-\Delta_{2,f} -\alpha)^{-m/2} e^{-\psi}\phi^{-1}].
\end{split}
\end{equation}
The perturbed operator will be bounded from $L^1$ to $L^\infty$  if we can show that
\begin{align*}
& \|e^{\psi} (-\Delta_{2,f} -\alpha)^{-m/2} e^{-\psi}\phi^{-1}\|_{L^1 \to L^2} \leq C \tag{a}\\
& \|e^{\psi} (-\Delta_{2,f} -\xi)^{-1} e^{-\psi}\|_{L^2 \to L^2} \leq C \tag{b}\\
& \|\phi^{-1} \, e^{\psi} (-\Delta_{2,f} -\alpha)^{-m/2} e^{-\psi}\|_{L^2 \to L^\infty} \leq C  \tag{c}
\end{align*}
for $m\geq N$.

Estimate (b) is Lemma \ref{LemKato} and determines the constant $\eps$.

Estimates (a) and (c) require the drifting heat kernel estimates. Given that $\alpha$ is a negative real number,  the operator  $(-\Delta_{2,f} -\alpha)^{-m/2}$ is given by
\[
(-\Delta_{2,f} -\alpha)^{-m/2} = C_m \int_0^\infty e^{t\Delta_{2,f}} \, t^{\frac{m}{2}-1} \, e^{\alpha t} \,dt
\]
where $C_m$ is a constant that depends only on $m$. Since $e^{t\Delta_{2,f}}$ has an integral kernel (which is the heat kernel), we observe that $(-\Delta_{2,f} -\alpha)^{-m/2}$ also has an integral kernel, $G^{m/2}_\alpha(x,y),$ whenever the right side of the above equation is integrable with respect to $t$. In particular, from equation \eref{heat42} of Corollary \ref{CorlT2b} for any $\beta>0$ and $m>n+q+2$, there exists an $\alpha<0$ such that
\[
|G^{m/2}_\alpha(x,y)| \leq C(n+q, K,\beta)\,\phi(x)^2 e^{-\beta\, d(x,y)}.
\]

As a result, $\phi^{-1} \, e^{\psi} (-\Delta_{2,f} -\alpha)^{-m/2} e^{-\psi}$ also has an integral kernel such that for any function $g\in L^2$ we get
\begin{equation}
\begin{split}
\| \phi^{-1} \, e^{\psi} (-&\Delta_{2,f} -\alpha)^{-m/2} e^{-\psi} g \|_{L^\infty}=\\
&= \sup_{y} \phi(y)^{-1} \, e^{\psi(y)} \bigl| \int_M G^{m/2}_\alpha(x,y) e^{-\psi(x)} \, g(x) \, e^{-f(x)}dv(x) \bigr|\\
&\leq \sup_{y} \bigl| \int_M \phi(y)^2  e^{2(\eps-\beta)d(x,y)}  \, e^{-f(x)}dv(x) \bigr|^{1/2}\, \|g(x)\|_{L^2} \\
&\leq C \|g(x)\|_{L^2}
\end{split}
\end{equation}
after choosing $\beta>0$ large enough, given the volume comparison result of Lemma \ref{VolCom}. Part (c) follows by taking adjoints. The upper bound on $ G^{m}_\xi(x,y)$  follows.

By Lemma \ref{VolCom2}, we know that  $(-\Delta_{2,f} -\xi)^{-m}$ is a bounded operator on $L^1(M)$. Since it is bounded  on  $L^2(M)$ by the definition of $\xi$, the interpolation theorem implies that the operator is bounded on $L^p(M)$ for any $1\leq p\leq 2$.  By taking adjoints, the operator is bounded for any $p\geq 2$. Finally,   the operator  is bounded on $L^\infty$ by our definition of the $L^\infty$ spectrum.
\end{proof}

\begin{proof}[Proof of Theorem {\rm \ref{ThmLp}}]
We begin by first proving the inclusion $\sigma(-\Delta_{p,f})\subset \sigma(-\Delta_{2,f})$. This is the more difficult one to achieve, since the reverse inclusion could be true in a more general setting. Proposition  \ref{ResProp} implies that  for any  $N\leq m\leq 2N$,
the operator $(-\Delta_{2,f} -\xi)^{-m}$  is  bounded on $L^p$ for all $p\in[1,\infty]$ and for all $\xi \in \rho(\Delta_{2,f})$. By the uniqueness of the  extension\footnote{We need to argue that $\xi$ can be path connected to a negative number, which is true since the spectrum of $\Delta_{2,f}$ is contained in the real line. For the details of this argument we refer the interested reader to the articles by Hempel and Voigt \cites{HemVo,HemVo2} (see also \cites{sturm,Char1}).} we get that $(-\Delta_{p,f} -\xi)^{-m}= (-\Delta_{2,f} -\xi)^{-m}$ for all $\xi \in \rho(-\Delta_{2,f})$. In what follows, we fix the number $N$.

If the inclusion $\sigma(-\Delta_{p,f})\subset \sigma(-\Delta_{2,f})$ is not true, since $\sigma(\Delta_{2,f})\subset [0,\infty)$, then there is a $\xi$ on the boundary of $\sigma(-\Delta_{p,f})$  which is not in $\sigma(-\Delta_{2,f})$. For any $\eps>0$, within an $\eps$ neighborhood of $\xi$, we can find $m-1$ different complex numbers $\xi_j$ such that
 $\xi_j\in \rho(-\Delta_{p,f})$. If $\eps$ is sufficiently small, using the Taylor expansion of $(-\Delta_{p,f}-\xi)^{-m}$, we have
  \begin{align}\label{expansion}
  \begin{split}
& \left((-\Delta_{p,f}-\xi)(-\Delta_{p,f}-\xi_1)\cdots (-\Delta_{p,f}-\xi_{m-1})\right)^{-1}\\
&=\sum_{k_1,\cdots, k_{m-1}\geq 0} (-\Delta_{p,f}-\xi)^{-m-k_1-\cdots-k_{m-1}}(\xi_1-\xi)^{k_1}\cdots(\xi_{m-1}-\xi)^{k_{m-1}}. \end{split} \end{align}
By Proposition~\ref{ResProp}, for any $k\geq N$, we have
\[
\|(-\Delta_{p,f}-\xi)^{-k}\|_{L^p\to L^p}= \|(-\Delta_{2,f}-\xi)^{-k}\|_{{L^p\to L^p}}\leq C{^{k/N}}.
\]
Therefore the expansion in~\eqref{expansion} is convergent and hence the operator is bounded on  $L^p$.
Using the partial fraction method, we can find non-zero complex numbers $c_j$ ($0\leq j\leq m-1$) such that
 \begin{align*}
 &
  \left((-\Delta_{p,f}-\xi)(-\Delta_{p,f}-\xi_1)\cdots (-\Delta_{p,f}-\xi_{m-1})\right)^{-1}\\
  &
  =c_0(-\Delta_{p,f}-\xi)^{-1}+\sum_j c_j(-\Delta_{p,f}-\xi_j)^{-1}.
  \end{align*}
 Therefore $(-\Delta_{p,f}-\xi)^{-1}$ is bounded on $L^p$, because all the other operators are bounded on $L^p$.

We will now show that $\sigma(-\Delta_{p,f})\supset \sigma(-\Delta_{2,f})$. Let $p$ and $q$ to be dual orders such that $p^{-1}+q^{-1}=1$ and observe that by duality $\rho(-\Delta_{p,f})=\rho(-\Delta_{q,f})$ (in other words, $\xi\in \rho(-\Delta_{p,f})$ if and only if $\xi\in\rho(-\Delta_{q,f})$. Therefore $(-\Delta_{p,f}-\xi)^{-1}$  and $(-\Delta_{q,f}-\xi)^{-1}$ are bounded operators on $L^p$ and $L^q$, respectively for all $\xi\in \rho(-\Delta_{p,f})$. Using the Calder\'on Lions Interpolation Theorem  \cite{ReSiII}*{Theorem IX.20}, Hempel and Voigt and Sturm prove that that the interpolated operator on $L^2$ is $(-\Delta_{2,f}-\xi)^{-1}$  \cites{HemVo2,sturm}. As a result, $(-\Delta_{2,f}-\xi)^{-1}$ is bounded on $L^2$ and $\xi\in \rho(-\Delta_{2,f})$.

 \end{proof}

\begin{corl}
In the notation above, for any $p\in[1,\infty]$, we have that the $L^p$ essential spectrum of $\Delta_\eps$ on $\tilde M_\eps$ contains the $L^p$ essential spectrum of $\Delta_f$. In particular, if the essential spectrum of $\Delta_f$ on $M$ is $[0,\infty)$, then so is the essential spectrum of $\Delta_\eps$ on $\tilde M_\eps$ for any $\eps$.

\end{corl}

\begin{proof}
Let $\mathfrak B$ be the Banach  space of real-valued functions on $\tilde M_\eps$ which are constants along the fiber and belong to $L^p(\tilde M_\eps)$. Then $\mathfrak B$ is isomorphic to $L^p(M, e^{-f}dv)$. Moreover, by  \eref{lap},  $\Delta_\eps$ can be identified to $\Delta_f$ under this isomorphism. Therefore, when restricted to $\mathfrak B$, the essential spectrum of $\Delta_\eps$ is the same  as that of $\Delta_f$ on $M$.
\end{proof}

\section{Manifolds Where the Essential Spectrum of $\Delta_f$ is $[0,\infty)$}\label{sec6}

In this section we will compute  the $L^p$ essential spectrum of the drifting Laplacian under adequate curvature conditions on the manifold.

As before we fix a point $x_o\in M$ and let $r(x)=d(x,x_o)$ be the radial distance to the point $x_o$. Then $r(x)$ is a continuous Lipschitz function and $|\n r|=1$ almost everywhere. The cut locus of $x_o\in M$ is a set of measure zero, denoted ${\rm Cut}(x_o)$. Then $M= S\cup {\rm Cut}(x_o) \cup \{x_o\} $ where $S$ is a star-shaped domain on which $r$ is  smooth.

We will prove the following theorem
\begin{thm} \label{thm62}
Let $(M^n,g,e^{-f} dv)$ be a weighted manifold. Suppose that for some $q>0$, the $q$-Bakry-\'Emery Ricci tensor ${\rm Ric}_f^q$  is asymptotically nonnegative in the radial direction with respect to a fixed point $x_o$ (cf. Definition~\ref{defnonnegative}). If the weighted volume of the manifold is finite, we additionally assume that its volume  does not decay exponentially at $x_o$.  Then the $L^2$ essential spectrum of the drifting Laplacian is $[0,\infty)$.
\end{thm}

By Lemma~\ref{L1}, the theorem is implied by the following more general result

\begin{thm} \label{thm63}
Let $(M^n,g,e^{-f} dv)$ be a weighted manifold. Suppose that, with respect to a fixed point $x_o$, the radial function $r(x)=d(x,x_o)$ satisfies
\begin{equation*}
{\limsup_{r\to \infty}}\; \Delta_f r\leq 0
\end{equation*}
in the sense of distribution. If the weighted volume of the manifold is finite, we additionally assume  that its volume  does not decay exponentially at $x_o$.  Then the $L^2$ essential spectrum of the drifting Laplacian  is $[0,\infty)$.
\end{thm}

We apply Proposition 3.1 of~\cite{char-lu-2} to find a smooth approximation $\tilde r$ to $r$ on $M$ with the same properties taking into account the weighted volume (see also~\cite{sil}). The proof of the above theorem resembles that of Theorem 1.1 in~\cite{char-lu-2}, once we set up the following estimate. We omit the details of the proof of both results.

Let $B (r)= B_{x_o}(r)$ be the ball of radius $r$ in $M$ at $x_o$ and denote its weighted volume by $V_f(r)$. As in the proof of Lemma 4.4~\cite{char-lu-2} (see also \cite{sil}), we get

\begin{lem} \label{DeltarVolEst}
Suppose that~\eqref{DeltaEst} is valid on $M$ in the sense of distribution.  Then we have the following two cases
\begin{enumerate}
\item [(a)]
Whenever  $\vol_f(M)$ is infinite, for any  $\eps>0$, and   $r_1>0$ {large enough},  there exists a $K=K(\eps, r_1)$ such that for any $r_2>K$,
we have
\begin{equation}\label{wsx-1}
\int_{B(r_2)\setminus B(r_1)} |\Delta_f \tilde{r}| \leq \eps \, \, V_f(r_2+1) + 2 ;
\end{equation}
\item [(b)]
Whenever  $\vol_f(M)$ is finite,  for any $\eps>0$ there exists a $K(\eps)>0$  such that for any $r_2>K$, we have
\[
\int_{M \setminus B(r_2)} |\Delta_f \tilde{r}| \leq  \eps \, \, (\vol(M)-V_f(r_2))+2  \vol(\pa B(r_2)).
\]
\end{enumerate}
\end{lem}

If we assume that $\textup{Ric}_f^q$ is asymptotically nonnegative, then we have the following stronger result for the $L^p$ spectra.

\begin{thm}\label{thm61}
Let $(M^n,g,e^{-f} dv)$ be a weighted manifold such that for some $q>0$
\[
\liminf_{r\to\infty} {\rm Ric}_f^q\geq 0.
\]
Then the $L^p$ essential spectrum of $\Delta_f$ is $[0,\infty)$ {for all $p\in[1,\infty]$}.
\end{thm}

By Theorem \ref{thm62} we know that the $L^2$ essential spectrum will be $[0,\infty)$ in the above case. In order to generalize the result for all $p$ by applying Theorem~\ref{ThmLp}, we need to prove the following volume comparison theorem which is of its own interest.

\begin{thm}
Under the assumptions of Theorem{\rm ~\ref{thm61}} the weighted volume of $M$ grows uniformly subexponentially.
\end{thm}

\begin{proof} We begin with some remarks on the volume comparison theorem Lemma~\ref{VolCom}. Firstly, using the same method as in the proof of the lemma, we conclude that if the curvature $\textup{Ric}_f^q$ is asymptotically nonnegative in the sense of Lemma~\ref{L1} at $x_o$, then the volume growth is  subexponential at $x_o$. Secondly, the volume comparison theorem can be localized. In other words, the volume comparison result of Lemma~\ref{VolCom} holds, whenever $\textup{Ric}_f^q\geq -K$ on the ball of radius $R$.

With the above remarks, the proof is essentially a compactness argument. Let $x_o\in M$ be a fixed point. If the theorem is not true, then there exists an $\eps_o>0$ such that for any $C>0$, there exist sequences $\{x_i\}\subset M$ and $\{R_i>0\}\subset \mathbb R$ such that $R_i\to \infty$ and
\begin{equation} \label{thm63e1}
C\,V_f(x_i,1)e^{\eps_o R_i} \leq V_f(x_i,R_i)
\end{equation}
for all $i\geq 0$. This implies that $d(x_i,x_o)\to \infty$ as $i\to \infty$. If not, then $V_f(x_i,1)$ has a uniform lower bound and the volume of the manifold grows exponentially by \eref{thm63e1}, which is a  contradiction.
At the same time Lemma \ref{VolCom} implies that inequality \eref{thm63e1} can only happen whenever
\begin{equation}\label{dist}
d(x_i,x_o)/ R_i\leq 2
\end{equation}
as $i\to\infty$. Otherwise, $B_{x_i}(R_i) \subset M\setminus B_{x_o}(R_i)$ and as a result, for any $\eps>0$ and $R_i$ large enough $\textup{Ric}_f^q\geq -\eps$ on $B_{x_i}(R_i)$. By Lemma \ref{VolCom} we have
\[
V_f(x_i,R_i)\leq C'\, V_f(x_i,1)e^{\eps  R_i}
\]
which is a clear contradiction to \eref{thm63e1} as $i\to \infty$.

Now applying \eref{dist} we have
\[
V_f(x_o, 3R_i)\geq V_f(x_i, R_i)\geq CV_f(x_i,1)e^{\eps_o R_i}.
\]

For any small $\eps>0$ we can find  $R_o>0$  a sufficiently large constant such that $\textup{Ric}_f^q\geq-\eps$ on  $M\setminus B_{x_o}(R_o)$. Fix  $\sigma>0$ and choose $i$  sufficiently large so that $\sigma+R_o <d(x_i,x_o)$. By Lemma \ref{VolCom}, for $\eps$  small enough and all $i$ sufficiently large we have
\[
V_f(x_i, d(x_i,x_o)-\sigma)\leq C'\,V_f(x_i,1) e^{\frac 12\eps_o(d(x_i,x_o)-\sigma)}.
\]
Let $y\in B_{x_i}(d(x_i,x_o)-\sigma-1)\cap B_{x_o}(\sigma+2)$. Then we have
\[
V_f(y,1)\leq V_f(x_i, d(x_i,x_o)-\sigma)\leq C'\,V_f(x_i,1) e^{\frac 1{4}\eps_o(d(x_i,x_o)-\sigma)}.
\]
On the other hand, the usual volume comparison theorem implies that
\[
V_f(x_o,1)\leq e^{\tilde C\sigma}V_f(y,1).
\]
Linking all the above inequalities, we have
\[
V_f(x_{o},3R_i)\geq C e^{\frac 12 \eps_o R_i}
\]
for all $i\geq 0$, which contradicts  the fact that the volume  of $M$ grows subexponentially with respect to one point.
\end{proof}

\begin{remark}
Let  $g$ be a smooth function of $M$. If we regard $g$ as a function on $\tilde M_\eps$, then we have
\[
\Delta_\eps g=\Delta_f g.
\]
Using this observation, we conclude that the $L^p$ spectrum of $\Delta_\eps$ on $\tilde M_\eps$ is $[0,\infty)$ whenever the $L^p$ spectrum of $\Delta_f$ on $M$ is $[0,\infty)$.
\end{remark}

\begin{remark}  It is interesting to observe that we can obtain information about the $L^2$ essential spectrum of the Laplacian  on a Riemannian  manifold with assumptions only on the Bakry-\'Emery Ricci tensor. For example, in the shrinking Ricci soliton case, where Bakry-\'Emery Ricci tensor is propositional to the Riemannian metric, we are able to compute the spectrum  of the Laplacian\cites{char-lu-2, Lu-Zhou_2011}.  Furthermore, we note that Silvares has demonstrated in \cite{sil} that the $L^2$ essential spectrum of the drifting Laplacian is $[0,\infty)$ whenever the Bakry-\'Emery Ricci  tensor, ${\rm Ric}_f$, is nonnegative and the weight function $f$ has sublinear growth.
\end{remark}



\end{document}